\documentclass[11pt]{amsart}
\usepackage[utf8]{inputenc}
\usepackage{bm}
\usepackage{fontenc}
\usepackage{amsfonts}
\usepackage{amssymb}
\usepackage{amsmath}
\usepackage{amsthm}\usepackage{mathtools}
\usepackage{enumerate}
\usepackage{hyperref}\usepackage{enumitem}
\usepackage{mathrsfs}
\usepackage{tikz}\usetikzlibrary{calc}
\usepackage{marginnote}
\usepackage{xcolor,enumitem}
\usepackage{soul}\usepackage{tikz}
\usepackage[all,cmtip]{xy}

\newcommand{\Z}{\mathbb{Z}}

\newcommand{\N}{\mathbb{N}}
\newcommand{\C}{\mathbb{C}}

\newcommand{\cE}{\mathcal{E}}

\newcommand{\eps}{\varepsilon}

\newcommand{\Hawaii}{Hawai\kern.05em`\kern.05em\relax i}

\newcommand{\SOTh}{\mathrm{SOT}\text{-}}

\newcommand{\cstu}{\mathrm{C}^*_u}

\newtheorem*{rigprob*}{Rigidity Problem for uniform Roe Algebras}
\newtheorem*{rigprobcorona*}{Rigidity Problem for uniform Roe Coronas}

\newcommand{\cst}{\mathrm{C}^*}
\newcommand{\cstar}{$\mathrm{C}^*$}

\newcommand{\cU}{\mathcal{U}}

\newcommand{\cF}{\mathcal{F}}
\newcommand{\cP}{\mathcal{P}}
\newcommand{\bbN}{\mathbb{N}}
\newcommand{\cB}{\mathcal{B}}
\newcommand{\cK}{\mathcal{K}}

\newtheorem{theorem}{Theorem}[section]
\newtheorem*{theorem*}{Theorem}
\newtheorem{proposition}[theorem]{Proposition}

\newtheorem*{proposition*}{Proposition}
\newtheorem{lemma}[theorem]{Lemma}
\newtheorem*{lemma*}{Lemma}
\newtheorem{corollary}[theorem]{Corollary}
\newtheorem*{corollary*}{Corollar}

\newtheorem*{fact*}{Fact}
\theoremstyle{definition}
\newtheorem{definition}[theorem]{Definition}
\newtheorem*{definition*}{Definition}
\newtheorem*{acknowledgments}{Acknowledgments}
\newtheorem{claim}[theorem]{Claim}
\newtheorem*{claim*}{Claim}

\newtheorem*{conjecture*}{Conjecture}

\newtheorem{question}[theorem]{Question}

\theoremstyle{remark}

\newtheorem*{example*}{Example}
\newtheorem{remark}[theorem]{Remark}
\newtheorem*{remark*}{Remark}

\newtheorem*{note*}{Note}
\newtheorem*{question*}{Question}

\newcommand{\norm}[1]{\left\lVert #1 \right\rVert}

\DeclareMathOperator{\supp}{supp}

\DeclareMathOperator{\propg}{prop}

\DeclareMathOperator{\diam}{diam}

\usepackage{enumitem}

\numberwithin{equation}{section}

\begin{document}

\title[ONL for equi-approximable families of projections]{Operator norm localization property for equi-approximable families of projections}%

\date{\today} 
 
\author[Braga]{Bruno M. Braga}
\address[B. M. Braga]{PUC-Rio, Departament of Mathematics,
Gavea, Rio de Janeiro - CEP 22451-900, Brazil}
\email{demendoncabraga@gmail.com}
\urladdr{https://sites.google.com/site/demendoncabraga}
 
\author[Farah]{Ilijas Farah}
\address[I. Farah]{Department of Mathematics and Statistics\\
York University\\
4700 Keele Street\\
North York, Ontario\\ Canada, M3J 1P3\\
and 
Matemati\v cki Institut SANU\\
Kneza Mihaila 36\\
11\,000 Beograd, p.p. 367\\
Serbia}
\email{ifarah@yorku.ca}
\urladdr{https://ifarah.mathstats.yorku.ca}

\author[Vignati]{Alessandro Vignati}
\address[A. Vignati]{
Institut de Math\'ematiques de Jussieu (IMJ-PRG)\\
Universit\'e Paris Cit\'e\\
B\^atiment Sophie Germain\\
8 Place Aur\'elie Nemours \\ 75013 Paris, France}
\email{alessandro.vignati@imj-prg.fr}
\urladdr{http://www.automorph.net/avignati}


\begin{abstract}
The rigidity problem for uniform Roe algebras was recently positively solved. Before its solution was found, there were positive solutions under the assumption of certain technical geometric conditions. In this paper, we introduce weaker versions of the operator norm localization property (ONL) which turn out to characterize those technical geometric conditions. We use this to obtain new rigidity results for nonmetrizable coarse spaces. As an application, we provide a novel partial answer to a question of White and Willett about Cartan subalgebras of uniform Roe algebras. We also study embeddings between uniform Roe algebras.
\end{abstract}
\maketitle

\section{Introduction}\label{SectionIntro}
Given a metric space $X$, its uniform Roe algebra, denoted by $\cstu(X)$, is a $\mathrm{C}^*$-subalgebra of $\cB(\ell_2(X))$ --- the space of bounded linear operators on $\ell_2(X)$ --- which captures several aspects of the coarse geometry of $X$ (we defer its formal definition to \S\ref{SectionPrelim}). The rigidity problem for uniform Roe algebras of uniformly locally finite metric spaces asked how ``rigid'' is the procedure of constructing $\cstu(X)$ given the metric space $X$. This problem has been recently positively solved: precisely, it was shown in \cite[Theorem 1.2]{BaudierBragaFarahKhukhroVignatiWillett2021uRaRig} that if $X$ and $Y$ are uniformly locally finite metric spaces with isomorphic uniform Roe algebras, then $X$ and $Y$ are coarsely equivalent.  Before this solution was found, several partial answers were obtained under some geometric conditions on the spaces. Its first partial solution was proven in \cite[Theorem 4.1]{SpakulaWillett2013AdvMath} under the assumption of the metric spaces having Yu's property A. This result was then strengthened in \cite[Theorem 6.1]{BragaFarah2018Trans} and proven to hold for metric spaces satisfying a technical condition on the ideal of ghost operators. The authors of \cite{LiSpakulaZhang2020} then observed that this technical condition could be weakened further (see \cite[Corollary 3.9]{LiSpakulaZhang2020}).\footnote{To the best of our knowledge, it is currently not know if the condition of \cite[Corollary 3.9]{LiSpakulaZhang2020} is actually \emph{strictly} weaker than the previously studied conditions. We provide a more detailed discussion of this question below. } 

These notes concern characterizations of the technical properties on the ideal of ghost operators mentioned above. We start by recalling those properties: given a metric space $(X,d)$, a subspace $X'\subseteq X$ is \emph{sparse} if $X'=\bigsqcup_nX_n$, where each $X_n$ is finite and $d(X_n,X_m)\to \infty$ as $n+m\to \infty$. A projection $p\in \cstu(X)$ is called a \emph{block projection} if there is a disjoint sequence $(X_n)_n$ of finite (non-empty) subsets of $X$ such that 
\[
p=\SOTh\sum_n\chi_{X_n}p\chi_{X_n}
\]
where $\chi_{X_n}$ is the orthogonal projection onto $\ell_2(X_n)$. If, moreover, each $\chi_{X_n}p\chi_{X_n}$ has rank 1, then $p$ is   called a \emph{block-rank-one projection}.
 
We can now formally state the aforementioned partial answers to the rigidity problem. Precisely, after the rigidity problem was solved for metric spaces with property A (\cite[Theorem 4.1]{SpakulaWillett2013AdvMath}), positive solutions were obtained for metric spaces satisfying the following properties:
\begin{enumerate} 
\item[(I)] If $X'\subseteq \cstu(X)$ is sparse, all ghost projections in $\cstu(X')$ are compact (\cite[Theorem 6.1]{BragaFarah2018Trans}).
\item[(II)] There are no block-rank-one ghost projections in $\cstu(X)$ (\cite[Corollary 3.9]{LiSpakulaZhang2020}).\footnote{Notice that the requirement of being block-rank-one implies the projection is not compact automatically.} 
\end{enumerate} 
It is immediate that (I) implies (II), however the question whether these two properties are equivalent is closely related to a tantalizingly innocent-looking question about the structure of uniform Roe algebras (see Question~\ref{Q1}). 

In these notes, we characterize each of the properties (I) and (II) in terms of natural weakenings of the operator norm localization property (ONL) --- a property known to be equivalent to property A (\cite[Theorem 4.1]{Sako2014} and \cite[Proposition 3.2]{BragaFarahVignati2020AnnInstFour}). We then use this characterization to obtain new rigidity results and new information about the Cartan masas of uniform Roe algebras; the latter provides a novel partial answer to a question of White and Willett (see Theorem \ref{Thm.CartanCosep}). 

Let us describe our main results. We start by recalling the definition of the operator norm localization property. While our definition is nonstandard, it is not difficult to see that it is equivalent to the usual one. The reason for our choice of a nonstandard definition will be clear in Definition \ref{DefinitionONLFamilies} below. If $(X,d)$ is a metric space and $a\in\mathcal B(\ell_2(X))$, the \emph{propagation} of $a$ is the quantity
\[
\propg(a)=\sup\{d(x,y)\mid \langle a\delta_x,\delta_y\rangle\neq 0\},
\]
 $(\delta_x)_{x\in X}$ being the canonical basis for $\ell_2(X)$.
\begin{definition}\label{DefinitionEquiApproxONL}
Let $(X,d)$ be a metric space. 
\begin{enumerate}
\item A family $S\subseteq \cstu(X)$ is \emph{equi-approximable} if for all $\eps>0$ there is $s>0$ such that for all $a\in S$ there is $b\in\cB(\ell_2(X))$ with $\propg(b)\leq s$ and $\|a-b\|\leq \eps$.
\item We say that $X$ has the \emph{operator norm localization property} (\emph{ONL}) if for all $\eps>0$ and all equi-approximable $S\subseteq \cstu(X)$, there is $s>0$ such that for all $a\in S$ there is a unit vector $\xi\in \ell_2(X)$ such that 
\[
\mathrm{diam}(\supp(\xi))\leq s\ \text{ and }\ \|a\xi\|\geq (1-\eps)\|a\|.
\]
\end{enumerate}
\end{definition} 
 
Restricting to families satisfying extra conditions, we obtain a family of natural weakenings of the ONL.
 
\begin{definition}\label{DefinitionONLFamilies}
Let $(X,d)$ be a metric space and let $\cP$ be a property of operators. We say that $X$ has the \emph{operator norm localization for $\cP$} (\emph{ONL for~$\cP$}) if for all $\eps>0$ and all equi-approximable families $S\subseteq \cstu(X)$ of operators satisfying $\cP$, there is $s>0$ such that for all $a\in S$ some unit vector $\xi\in \ell_2(X)$ satisfies $\mathrm{diam}(\supp(\xi))\leq s$ and $\|a\xi\|\geq (1-\eps)\|a\|$.
\end{definition} 

In the following two theorems, we characterise properties (I) and (II) stated above. 
The following are proven as Theorems \ref{ThmBlockProjIFFONLproj} and \ref{ThmBlockRank1ProjIFFONLproj}.  They are natural analogs of the well-known fact (see \cite{RoeWillett2014} and \cite{Sako2014}) that the classical version of ONL is equivalent to the non-existence of non-compact ghost projections.

\begin{theorem}\label{ThmBlockProjIFFONLprojINTRO}
Let $(X,d)$ be a uniformly locally finite metric space. The following are equivalent:
\begin{enumerate}
\item 
There are no non-compact block ghost projections in $\cstu(X)$. 
\item
If $X'\subseteq \cstu(X)$ is sparse, then all ghost projections in $\cstu(X')$ are compact.
\item
$X$ has ONL for equi-approximable finite-rank projections.
\end{enumerate}
\end{theorem} 

\begin{theorem}\label{ThmBlockRank1ProjIFFONLprojINTRO}
Let $(X,d)$ be a uniformly locally finite metric space. The following are equivalent:
\begin{enumerate}
\item There are no non-compact block rank one ghost projections in $\cstu(X)$. 
\item $X$ has ONL for equi-approximable rank-one projections.
\end{enumerate}
\end{theorem}

Theorems \ref{ThmBlockProjIFFONLprojINTRO} and \ref{ThmBlockRank1ProjIFFONLprojINTRO} show that weak requirements on the ideal of ghost operators are already strong enough to force the space to satisfy ONL in some sense. This may provide an approach to extending results previously known to hold only under ONL to more general spaces. 
 
We point out that the equivalent conditions in Theorem \ref{ThmBlockRank1ProjIFFONLprojINTRO} above are also equivalent to ``$X$ contains no sparse subspaces consisting of ghostly measured asymptotic expanders'' (see \cite[Corollary C]{LiSpakulaZhang2020}), and thus can be geometrically characterized.  We do not know of a geometric characterization of the conditions in Theorem \ref{ThmBlockProjIFFONLprojINTRO}.

\subsection{Application to rigidity}

Our first application is to rigidity of uniform Roe algebras of coarse (not necessarily metrizable) spaces. Indeed, coarse spaces generalize the concept of metric spaces in the context of coarse geometry. We refer the reader to \S\ref{SectionPrelim} for precise definitions. For now, we simply say that a coarse space is a set $X$ together with a family $\cE$ of subsets of $X\times X$ which gives a notion of uniform boundedness of families of subsets of $X$. The uniform Roe algebra of $(X,\cE)$ is, as in the case when $X$ is metric, denoted by $\cstu(X)$. While \cite[Theorem 1.2]{BaudierBragaFarahKhukhroVignatiWillett2021uRaRig} showed that if $(X,d)$ and $(Y,\partial)$ are uniformly locally finite metric spaces, then $X$ and $Y$ must be coarsely equivalent provided that $\cstu(X)$ and $\cstu(Y)$ are isomorphic, the same problem remains open if $(X,\cE)$ and $(Y,\cF)$ are uniformly locally finite \emph{coarse} spaces.

The next result was only known to hold in case of one of the spaces of interest has property A (see \cite[Theorem 1.3]{BragaFarahVignati2020AnnInstFour}).

\begin{theorem}\label{Thm.PreservationMetric}
Let $(X,\cE)$ and $(Y,\cF)$ be uniformly locally finite coarse spaces with $\cstu(X)\cong\cstu(Y)$. If $X$ is metrizable and $\cstu(X)$ has no block-rank-one ghost projections, then $(X,\cE)$ and $(Y,\cF)$ are coarsely equivalent. In particular, $(Y,\cF)$ is metrizable.

Moreover, a coarse equivalence is given by some (any) $g\colon Y\to X$ for which some $\delta>0$ satisfies $\|\Phi^{-1}(\chi_{y})\delta_{g(y)}\|>\delta$ for all $y\in Y$. 
\end{theorem}

\subsection{Applications to Cartan masas} The results of this section are concerned with metric spaces and the following notion. 

\begin{definition}\label{Def.Cartan}
Let $A$ be a unital \cstar-algebra and $B\subseteq A$ be a \cstar-subalgebra. We say that~$B$ is a \emph{Cartan masa} in $A$ if
\begin{enumerate}
\item $B$ is a maximal abelian self-adjoint subalgebra (masa) of $A$,
\item $A$ is generated as a \cstar-algebra by the \emph{normalizer} of $B$ in $A$, i.e., 
\[N_{A}(B)=\{a\in A\mid aBa^*\cup a^*Ba\subseteq B\},\]
\item there is a faithful conditional expectation $\Upsilon\colon A\to B$. 
That is, for all $a\in A$ it satisfies the following conditions: 
\begin{enumerate}
\item $\Upsilon\colon A\to B$ is completely positive. 
\item $\Upsilon(b_1ab_2)=b_1\Upsilon(a)b_2$ for all $b_1,b_2$ in $B$. 
\item $\Upsilon(a^*a)=0$ implies $a=0$. 
\end{enumerate}
\end{enumerate}
In the terminology of \cite{WhiteWillett2017}, a Cartan masa $B\subseteq \cstu(X)$ is \emph{co-separable} if there is a countable $S\subseteq \cstu(X)$ such that $\cstu(X)=\mathrm{C}^*(B,S)$. 
\end{definition}

Note that the maximality of $B$ implies $1_A\in B$ (in the non-unital case the definition of a Cartan masa involves an approximate unit). 
It is not difficult to see that for every uniformly locally finite space $X$, $\ell_\infty(X)$ is a Cartan masa in $\cstu(X)$;
the (unique) conditional expectation is defined by (see \S\ref{S.uRA} for the notation $\chi_A$)
\[
E(a)=\sum_{x\in X} \chi_{x} a\chi_{x}. 
\]
Conversely, by the duality result established in \cite[Theorem 4.17]{WhiteWillett2017}, if a unital \cstar-subalgebra of $\cB(\ell_2(X))$ contains $\mathcal{K}(\ell_2(X))$ and has a Cartan masa $B$ isomorphic to~$\ell_\infty(X)$ then $X$ carries a uniformly locally finite coarse structure such that $A$ is naturally isomorphic to $\cstu(X)$, via an isomorphism that sends $B$ to $\ell_\infty(X)$. 
In \cite[Remark 3.4]{WhiteWillett2017}, the authors ask whether every Cartan masa of $\cstu(X)$, where $(X,d)$ is a metric space, is automatically co-separable. 
In other words, by \cite[Theorem 4.17]{WhiteWillett2017} this question asks whether $\cstu(X)$ and $\cstu(Y)$ can be isomorphic for uniformly locally finite spaces $X$ and $Y$ if exactly one of them is a metric space. Until now, the strongest result in this direction was \cite[Theorem~1.12]{BaudierBragaFarahKhukhroVignatiWillett2021uRaRig}. This result asserts that if $\cstu(X)$ and $\cstu(Y)$ are isomorphic, both spaces are uniformly locally finite, and $X$ is metrizable, then $Y$ is countable and it contains a coarse copy of $X$. 

The next result was only known to hold under the additional assumption that $X$ has  property A (see \cite[Theorem 1.3]{BragaFarahVignati2020AnnInstFour}).
Note that the assumption on the absence of ghost projections is strictly weaker than property A \cite[p. 1010]{BragaFarah2018Trans}. 

\begin{theorem}\label{Thm.CartanCosep}
Let $X$ be a uniformly locally finite metric space and assume that $\cstu(X)$ has no block-rank-one ghost projections. If $B\subseteq \cstu(X)$ is a Cartan masa isomorphic to $\ell_\infty(\N)$, then $B$ is co-separable in $\cstu(X)$.
 
 \end{theorem}
 
Still in the topic of Cartan masas, the importance of the question whether a Cartan masa $B\subseteq \cstu(X)$ isomorphic to $\ell_\infty(\N)$ can contain noncompact ghosts to the rigidity problem for uniform Roe algebra has been known for a while (see the paragraph after \cite[Theorem 6.2]{BragaFarah2018Trans}). We show the following:

\begin{theorem}\label{Thm.GhostsAndCosep}
Let $(X,d)$ be a uniformly locally finite metric space and $B\subseteq \cstu(X)$ be a Cartan masa isomorphic to $\ell_\infty(\N)$. The following are equivalent:
\begin{enumerate}
\item\label{Thm.GhostsAndCosep.Item1} All ghost operators in $B$ are compact. 
\item\label{Thm.GhostsAndCosep.Item2} All ghost projections in $B$ are compact. 
\item\label{Thm.GhostsAndCosep.Item3} $B$ is co-separable in $\cstu(X)$.
\end{enumerate} 
\end{theorem}
 
 By \cite[Theorem~1.2]{BaudierBragaFarahVignatiWillett2022vNA}, $L_\infty[0,1]$ does not embed into $\cstu(X)$ and therefore the assumption that $B$ is isomorphic to $\ell_\infty(\bbN)$ can be weakened to $B$ being a von Neumann algebra. 

Using the duality result \cite[Theorem 4.17]{WhiteWillett2017} again, this puts \cite[Theorem 1.12]{BaudierBragaFarahKhukhroVignatiWillett2021uRaRig} in proper context. 
 
\subsection{A word on embeddings} 
Finally, in \S\ref{SectionEmb}, we deal with embeddings between uniform Roe algebras, a study of which was initiated in \cite{BragaFarahVignati2019Comm}. The existence of such embeddings often suffices to guarantee the existence of nontrivial maps $X\to Y$, which in turn imposes restrictions on the geometry of $X$ given by the geometry of $Y$. Using techniques from \cite{BaudierBragaFarahKhukhroVignatiWillett2021uRaRig}, we further develop the theory of embeddings between uniform Roe algebras (see Theorems \ref{ThmEmbCanonicalMASAToMASA} and \ref{ThmEmb1} for details).

\section{Preliminaries}\label{SectionPrelim}

\subsection{Basic definitions}
\subsubsection{Coarse spaces} 
We start by recalling the definition of a coarse space --- we refer the reader to \cite[Chapter~2]{RoeBook} for a detailed treatment of the subject. Loosely speaking, coarse spaces are abstractions of metric spaces which still allow one to talk about large-scale geometry. Precisely, let $X$ be a set and $\cE$ be a family of subsets of $X\times X$. We say that $\cE$ is a \emph{coarse structure on $X$} if
\begin{enumerate}
\item $\Delta_X=\{(x,x)\in X\times X\mid x\in X\}$ belongs to $\cE$,
\item if $E\in \cE$ and $F\subseteq E$, then $F\in \cE$,
\item if $E,F\in \cE$, then $E\cup F\in \cE$, 
\item if $E\in\cE$, then $E^{-1}=\{(y,x)\in X\times X\mid (x,y)\in E\}$
is in~$\cE$ and, 
\item if $E, F\in\cE$, then 
\[
E\circ F=\{(x,y)\in X\times X\mid \exists z\in X, \ (x,z)\in E\text{ and } (z,y)\in F\}
\]
belongs to $\cE$. 
\end{enumerate}
The elements of $\cE$ are called \emph{controlled sets} or \emph{entourages}. The pair $(X,\cE)$ is then called a \emph{coarse space}. Metric spaces have a canonical coarse structure: if $(X,d)$ is a metric space, 
\[
\cE_d=\Big\{E\subseteq X\times X\mid \sup_{(x,y)\in E}d(x,y)<\infty\Big\}
\]
is a coarse structure. Throughout this paper, metric spaces are viewed as coarse spaces with the coarse structure described above. A coarse space $(X,\cE)$ is called \emph{metrizable} when $\cE=\cE_d$ for some metric $d$ on $X$.

Let $(X,\cE)$ and $(Y,\cF)$ be coarse spaces and $f\colon X\to Y$ be a map. We say that $f$ is \emph{coarse} if for all $E\in \cE$ there is $F\in \cF$ such that 
\[
(x,z)\in E\ \text{ implies}\ (f(x),f(z))\in F,
\]
and we say that $f$ is \emph{expanding} if for all $F\in \cF$ there is $E\in \cE$ such that 
\[
(x,z)\not\in E\ \text{ implies } \ (f(x),f(z))\not\in F.
\]
If $f$ is both coarse and expanding, then $f$ is a \emph{coarse embedding}. If $f$ is a coarse embedding and there is $F\in \cF$ such that for all $y\in Y$ there is $x\in X$ with $(f(x),y)\in F$, the map $f$ is called a \emph{coarse equivalence}.

A coarse space $(X,\cE)$ is \emph{uniformly locally finite} (abbreviated as \emph{u.l.f.}) if for all $E\in \cE$ we have that
\[
\sup_{x\in X}|\{z\in X\mid (x,z)\in E\}|<\infty.
\]
For a metric space $(X,d)$, this simply means that for each $r>0$ there is $N>0$ such that every $r$-ball in $X$ has at most $N$ elements. 

\subsubsection{Uniform Roe algebras} \label{S.uRA}
Given a Hilbert space $H$, $\cB(H)$ denotes the space of bounded operators on $H$. Given a set $X$, $\ell_2(X)$ denotes the Hilbert space of square-summable functions $X\to \C$ and we denote its canonical orthonormal basis by $(\delta_x)_{x\in X}$. Given $A\subseteq X$, $\chi_A$ denotes the orthogonal projection onto $\ell_2(A)$. If $x\in X$, we write $\chi_x$ for $\chi_{\{x\}}$.

Given an operator $a\in\cB(\ell_2(X))$ and $x,z\in X$, we let $a_{x,z}=\langle a\delta_z,\delta_x\rangle$; we identify $a$ with the $X\times X$ matrix $[a_{x,z}]_{x,z\in X}$. If $(X,\cE)$ is a coarse space, we say that an operator $a=[a_{x,z}]\in \cB(\ell_2(X))$ has \emph{controlled propagation} if
\[
\mathrm{supp}(a)=\{(x,z)\in X\times X\mid a_{x,z}\neq 0\}
\]
belongs to $\cE$. 
\begin{definition}
Let $(X,\cE)$ be a coarse space. The norm closure of all operators $a\in \cB(\ell_2(X))$ with controlled propagation is the \emph{uniform Roe algebra of $(X,\cE)$}, denoted by $\cstu(X)$.
\end{definition}

We identify $\ell_\infty(X)$ with the set of all operators in $\mathcal B(\ell_2(X))$ diagonalized by the canonical basis. These coincide with the operators $a\in \cB(\ell_2(X))$ such that $\mathrm{supp}(a)\subseteq \Delta_X$. Hence every uniform Roe algebra contains $\ell_\infty(X)\subseteq \cB(\ell_2(X))$. If $(X,\cE)$ is a \emph{connected} coarse space, i.e., if $\{(x,z)\}\in \cE$ for all $x,z\in X$, then $\cstu(X)$ also contains the compact operators; this is always the case if $X$ is metrizable.
 
The definitions of projections in $\cB(\ell_2(X))$ being block projections or block-rank-one projections for coarse spaces are identical to the ones for metric spaces given in the introduction. However, to make sense of sparseness of a subspace $X'\subseteq X$, one needs the coarse structure to be countably generated, which is essentially the same as metrizability of the coarse structure \cite[Section 2.4]{RoeBook}.  Hence, we only talk about projections on sparse subspaces of $X$ when $X$ is metrizable. 

\subsection{Ghosts, equi-approximability, and ONL}
In this subsection, we present some extra technical definitions about uniform Roe algebras. Let $(X,\cE)$ be a u.l.f.\ coarse space. 
An operator $a\in \cstu(X)$ is a \emph{ghost} if for all $\eps>0$ there is a finite $A\subseteq X$ such that $|a_{x,z}|\leq \eps$ for all $x,z\not\in A$. Clearly, compacts operators in $\cstu(X)$ are always ghosts.

We generalise the definitions of equi-approximability and ONL from metric to coarse spaces (cf. Definitions \ref{DefinitionEquiApproxONL} and \ref{DefinitionONLFamilies}).

\begin{definition} 
Let $(X,\cE)$ be a u.l.f.\ coarse space.

\begin{enumerate}
\item Let $\eps>0$, $E\in \cE$, and $a\in \cB(\ell_2(X))$. We say that $a$ is $\eps$-$E$-approximable if there is $b\in \cB(\ell_2(X))$ with $\supp(b)\subseteq E$ such that $\|a-b\|\leq \eps$.
\item A family $S\subseteq\cstu(X)$ is \emph{equi-approximable} if for all $\eps>0$ there is $E\in \cE$ such that each $a\in S$ is $\eps$-$E$-approximable.\end{enumerate}
\end{definition}
 
The next lemma is fundamental for the results of the present paper. It will be used to guarantee that certain families of operators are equi-approximable. By $\SOTh\sum_{n\in M}p_n$ we denote the SOT-limit of finite partial sums, and by using this notation we indicate that this limit exists. 

\begin{lemma}[{\cite[Lemma 4.9]{BragaFarah2018Trans}}]
Let $(X,d)$ be a u.l.f. metric space and let $(p_n)_n$ be a sequence of orthogonal projections such that $\SOTh\sum_{n\in M}p_n$ belongs to $\cstu(X)$ for all $M\subseteq \N$. Then $(\SOTh\sum_{n\in M}p_n)_{M\subseteq \N}$ is equi-approximable.\label{LemmaUnifApprox}
\end{lemma}

\begin{definition}
Let $(X,\cE)$ be a coarse space.
 
\begin{enumerate}
\item Given $a\in \cB(\ell_2(X))$, $\eps>0$, and $E\in \cE$, we say that $a$ is \emph{$(\eps,E)$-normed} if there is $A\subseteq X$ with $A\times A\subseteq E$ such that $\|a\chi_A\|\geq (1-\eps)\|a\|$.
 \item We say that $X$ has the \emph{operator norm localization property} (ONL) if for all $\eps>0$ and all equi-approximable $S\subseteq \cstu(X)$, there is $E\in \cE$ such that every $a\in S$ is $(\eps,E)$-normed.
 \item Let $\cP$ be a property of operators. We say that $(X,\cE)$ has the \emph{operator norm localization property for $\cP$} (ONL for $\cP$) if for all $\eps>0$ and all equi-approximable families $S\subseteq \cstu(X)$ of operators satisfying $\cP$, there is $E\in\cE$ such that every $a\in S$ is $(\eps,E)$-normed.
\end{enumerate}
\end{definition}

\subsection{Candidates for coarse equivalences} The content of this subsection will be used for the applications of our main results in \S\ref{SectionCoSep}.
Let $(X,\cE)$ and $(Y,\cF)$ be coarse spaces. Given an isomorphism $\Phi\colon \cstu(X)\to \cstu(Y)$, a natural candidate for a coarse equivalence (or embedding) is a map $f\colon X\to Y$ such that 
\begin{equation}\label{EqfWitRig}\inf_{x\in X}\|\Phi(\chi_x)\delta_{f(x)}\|>0.\tag{$\ast$}
\end{equation} It is therefore necessary to understand when such maps exist. The following is one of the main results of \cite{BaudierBragaFarahKhukhroVignatiWillett2021uRaRig} and it can be extracted from the proof of \cite[Theorem 1.2]{BaudierBragaFarahKhukhroVignatiWillett2021uRaRig}. 

\begin{corollary}\label{CorRigPaper}
Let $(X,\cE)$ and $(Y,\cF)$ be u.l.f.\ coarse spaces and let $\Phi\colon \cstu(X)\to \cstu(Y)$ be an isomorphism. If $(X,\cE)$ is metrizable, then there is $f\colon X\to Y$ satisfying \eqref{EqfWitRig}.
\end{corollary}

Since this result hasn't been stated explicitly in \cite{BaudierBragaFarahKhukhroVignatiWillett2021uRaRig}, for the reader's convenience, we include its brief proof below. First, a lemma.
\begin{lemma}[{\cite[Corollary 3.3]{BaudierBragaFarahKhukhroVignatiWillett2021uRaRig}}]\label{LemmaCorRigPaper}
 Let $(X,\cE)$ be a metrizable u.l.f.\ coarse space and let $(p_n)_{n\in \N}$ be a sequence of projections in $\mathcal{B}(\ell_2(X))$ such that 
\begin{enumerate}
\item the family $(\sum_{n\in A}p_n)_{A\subseteq \N,|A|<\infty}$ is equi-approximable, and 
\item $\SOTh\sum_{n\in\N}p_n=1_{\ell_2(X)}$.
\end{enumerate}
Then, 
\[
\inf_{x\in X}\sup_{n\in\N} \|p_n\delta_x\|>0.
\]
\end{lemma}
 
\begin{proof}[Proof of Corollary \ref{CorRigPaper}]
As isomorphisms between uniform Roe algebras are strongly continuous (\cite[Lemma 3.1]{SpakulaWillett2013AdvMath}), the projections $(\Phi^{-1}(\chi_y))_{y\in Y}$ satisfy the hypothesis of Lemma \ref{LemmaCorRigPaper}. As $X$ is metrizable, this lemma applies and there is $\delta>0$ and $f\colon X\to Y$ such that $\|\Phi^{-1}(\chi_{f(x)})\delta_{x}\|> \delta$ for all $x\in X$. Therefore, for all $x\in X$,
\[
\|\Phi(\chi_x)\delta_{f(x)}\|=\|\Phi(\chi_x)\chi_{f(x)}\|=\|\chi_x\Phi^{-1}(\chi_{f(x)})\|=\|\Phi^{-1}(\chi_{f(x)})\delta_{x}\|> \delta.\qedhere
\]
\end{proof}

The next lemma highlights the reason why a map $f\colon X\to Y$ satisfying \eqref{EqfWitRig} is important for rigidity problems. Item \eqref{LemmaMapWitRig1} is \cite[Lemma 5.2]{BragaFarahVignati2019Comm} and items \eqref{LemmaMapWitRig2} and \eqref{LemmaMapWitRig3} can be extracted from the proof of \cite[Theorem~1.12]{BaudierBragaFarahKhukhroVignatiWillett2021uRaRig}.

A function is said to be \emph{uniformly finite-to-one} if there is a uniform finite bound on the cardinalities of the preimages of points in its range. 

\begin{lemma}\label{LemmaMapWitRig}
Let $(X,\cE)$ and $(Y,\cF)$ be u.l.f.\ coarse spaces and suppose $(X,\cE)$ is metrizable. Let $\Phi\colon \cstu(X)\to \cstu(Y)$ be an embedding and suppose that $f\colon X\to Y$ satisfies \eqref{EqfWitRig}. The following holds:
\begin{enumerate}
\item \label{LemmaMapWitRig1}If $(Y,\cF)$ is metrizable, then $f$ is coarse and uniformly finite-to-one.\footnote{It is straightforward to check that any map satisfying \eqref{EqfWitRig} is uniformly finite-to-one; see \cite[Lemma 5.1]{BragaFarahVignati2019Comm} for details.}
\item \label{LemmaMapWitRig2}If $c_0(Y)\subseteq \Phi(\cstu(X))$, then $f$ is coarse and uniformly finite-to-one. 
\item \label{LemmaMapWitRig3} If $\Phi$ is surjective, then $f$ is a coarse embedding.
\end{enumerate}
\end{lemma}

\section{ONL for equi-approximable projections}\label{SectionONLEquiApprox}
 
In this section, we prove Theorems \ref{ThmBlockProjIFFONLprojINTRO} and \ref{ThmBlockRank1ProjIFFONLprojINTRO}. Both of those results follow immediately from the next two more technical theorems. 
 
\begin{theorem}\label{ThmBlockProjIFFONLproj}
Let $(X,\cE)$ be a u.l.f.\ coarse space and consider the following assertions. 
\begin{enumerate}
\item \label{ItemThmBlockProjIFFONLproj2.5} All ghost block projections in $\cstu(X)$ are compact.
\item\label{ItemThmBlockProjIFFONLproj3} $X$ has ONL for finite rank projections.
\end{enumerate}
Then, \eqref{ItemThmBlockProjIFFONLproj2.5}$\Rightarrow$\eqref{ItemThmBlockProjIFFONLproj3}. Moreover, if $(X,\cE)$ is metrizable, then \eqref{ItemThmBlockProjIFFONLproj3}$\Rightarrow$\eqref{ItemThmBlockProjIFFONLproj2.5} and those conditions are also equivalent to the following:
\begin{enumerate}\setcounter{enumi}{2}
\item\label{ItemThmBlockProjIFFONLproj1} If $X'\subseteq X$ is sparse, then all ghost projections in $\cstu(X')$ are compact.
\item \label{ItemThmBlockProjIFFONLproj2} If $X'\subseteq X$ is sparse, then all ghost block projections in $\cstu(X')$ are compact.
\end{enumerate}
\end{theorem}
 
\begin{theorem}\label{ThmBlockRank1ProjIFFONLproj}
Let $(X,\cE)$ be a u.l.f.\ coarse space and consider the following assertions.
\begin{enumerate}
\item There are no ghost block-rank-one projections in $\cstu(X)$. 
\item $X$ has ONL for rank one projections.
\end{enumerate}
Then, \eqref{ItemThmBlockProjIFFONLproj2.5}$\Rightarrow$\eqref{ItemThmBlockProjIFFONLproj3}. Moreover, if $(X,\cE)$ is metrizable, then \eqref{ItemThmBlockProjIFFONLproj3}$\Rightarrow$\eqref{ItemThmBlockProjIFFONLproj2.5} and those conditions are also equivalent to the following:
\begin{enumerate}\setcounter{enumi}{2}
\item If $X'\subseteq X$ is sparse, then there are no noncompact ghost block-rank-one projections in $\cstu(X')$.
\end{enumerate}
\end{theorem}

Before presenting the (very similar) proofs of Theorems \ref{ThmBlockProjIFFONLproj} and \ref{ThmBlockRank1ProjIFFONLproj}, we need some lemmas. The next technical lemma was proved in \cite[Lemma 5.4]{BaudierBragaFarahKhukhroVignatiWillett2021uRaRig} for metrizable u.l.f.\ coarse spaces. Its proof for an arbitrary coarse space is virtually identical, so we omit it here. 
 
 \begin{lemma}[{\cite[Lemma 5.4]{BaudierBragaFarahKhukhroVignatiWillett2021uRaRig}}]
 Let $(X,\cE)$ be a u.l.f. coarse space. Then, given $\eps,\delta>0$ there is $\gamma>0$ so that for all $E,F\in \cE$ there is $W\in \cE$ for which the following holds: let $p,q,a\in \cB(\ell_2(X))$, where $p$ is a projection and $q$ is a rank 1 projection. If $\supp(q)\subseteq E$, $\|p-a\|<\gamma $, $\supp(a)\subseteq F $, and $\|pq\|\geq \delta$, then $p$ is $(\eps,W)$-normed.\label{LemmaSortOfONLForAnySpace}
\end{lemma}

The previous lemma allows us to show that, given $\eps>0$, certain equi-approximable families can be $(\eps,W)$-normed with respect to a single entourage $W\in \cE$.
 
\begin{lemma}\label{LemmaSortONLForEquiApproxFamOfProj}
Let $X$ be a u.l.f.\ coarse space and let $\cP$ be an equi-approximable family of finite rank projections in $\cstu(X)$. The following are equivalent:
\begin{enumerate}
\item\label{item1SortONL}$\inf_{p\in \cP }\sup_{x\in X}\|p\chi_{x}\|>0$.
\item\label{item2SortONL}For all $\eps>0$, there is $W\in \cE$ such that each $p\in \cP$ is $(\eps,W)$-normed 
(that is, $(X,d)$ satisfies ONL for finite rank projections).
\item\label{item3SortONL} There are $\eps\in (0,1)$ and $W\in \cE$ such that each $p\in \cP$ is $(\eps,W)$-normed. 
\end{enumerate}
\end{lemma}

\begin{proof}
\eqref{item1SortONL}$\Rightarrow$\eqref{item2SortONL}: Fix $\eps>0$. Pick $\delta>0$ and a family $(x_p)_{p\in \cP}$ in $X$ such that $\|p\chi_{x_p}\|>\delta$ for all $p\in \cP $. Let $\gamma>0$ be given by Lemma \ref{LemmaSortOfONLForAnySpace} for $\eps$ and $\delta$. As $\cP$ is equi-approximable, there is $E\in \cE$ such that every $p\in\cP$ is $\gamma$-$E$-approximable. Let $W\in \cE$ be given by Lemma \ref{LemmaSortOfONLForAnySpace} for $E$ and $F=\Delta_X$. 

By our choice of $E$, for each $p\in \cP $, there is $a_p\in \cstu(X)$ with $\supp(a_p)\leq E$ such that $\|p-a_p\|\leq \gamma$. By Lemma \ref{LemmaSortOfONLForAnySpace} applied with $q$ as the projection to the span of~$\delta_{x_p}$, each $p$ is $(\eps,W)$-normed.

\eqref{item2SortONL}$\Rightarrow$\eqref{item3SortONL} is obvious, so we are left with
\eqref{item3SortONL}$\Rightarrow$\eqref{item1SortONL}. Let $\eps\in (0,1)$ and $W$ such that each $p\in\cP$ is $(\eps,W)$-normed. Since $X$ is u.l.f., there is $k$ such that if $A\times A\subseteq W$ then $|A|\leq k$, for all $A\subseteq X$. For each $p\in\cP$, pick $A_p$ with $\norm{p\chi_{A_p}}\geq (1-\eps)$ and $A_p\times A_p\subseteq W$. Then we can find $x\in A_p$ such that $\norm{p\chi_x}\geq\frac{1-\eps}{k}$. As $\eps$ and $k$ are fixed, this completes the proof.
\end{proof}
 
A well-known example shows that it is not true in general that, given an arbitrary u.l.f.\ coarse space $(X,\cE)$, any equi-approximable family of finite rank projections $\cP$ in $\cstu(X)$ must satisfy 
\[
\inf_{p\in \cP }\sup_{x\in X}\|p\chi_{x}\|>0.
\]
Indeed, if the space $(X,d)$ contains a subspace on which the metric is the graph metric given by a sequence $X_n$ of expander graphs (\cite{lubotzky2012expander}), then there is a disjoint sequence $(X_n)_n$ of finite subsets of $X$ such that for every~$n$ 
\[
p_n=\begin{bmatrix}
\frac{1}{|X_n|}&\ldots & \frac{1}{|X_n|}\\
\vdots&\ddots&\vdots\\
\frac{1}{|X_n|}&\ldots&\frac{1}{|X_n|} \end{bmatrix}\in \cB(\ell_2(X_n))\subseteq \cB(\ell_2(X))
\]
defines a projection in $\cB(\ell_2(X_n))$. The spectral gap property of the discrete Laplacian operator associated with the expanders implies that for each $M\subseteq \N$ 
\[
\SOTh\sum_{n\in M}p_n\in \cstu(X), 
\] 
and moreover that this family is equi-approximable as $M$ varies over all subsets of $\N$ (see \cite[pp. 348--349]{HigsonLafforgueSkandalis2002GAFA}).  On the other hand, $\inf_{n\in\N }\sup_{x\in X}\|p_n\chi_x\|=0$. 

This example indicates the need for a workable condition that implies that a family of equi-approximable projections $\cP$ satisfies $\inf_{p\in \cP }\sup_{x\in X}\|p\chi_x\|>0$. This is the content of the next lemma. 

\begin{lemma}\label{LemmaGeoCondForEquiApproxAreBoundedBelow}
Let $X$ be a u.l.f.\ coarse space and let $\cP$ be an equi-approximable family of nonzero finite-rank projections in $\cstu(X)$.

\begin{enumerate}
\item \label{LemmaGeoCondForEquiApproxAreBoundedBelow1} If all ghost block projections in $\cstu(X)$ are compact, then 
\[
\inf_{p\in \cP}\sup_{x\in X}\|p\chi_x\|>0.
\]
\item \label{LemmaGeoCondForEquiApproxAreBoundedBelow2} If all ghost block-rank-one projections in $\cstu(X)$ are compact and each $p\in \cP$ has rank one, then 
\[
\inf_{p\in \cP}\sup_{x\in X}\|p\chi_x\|>0.
\]
\end{enumerate}
\end{lemma} 
 
\begin{proof}We will prove \eqref{LemmaGeoCondForEquiApproxAreBoundedBelow1} and \eqref{LemmaGeoCondForEquiApproxAreBoundedBelow2} simultaneously, indicating the only difference in the proofs in the appropriate moment. By contradiction, pick, for each $n\in\N$, $p_n\in\cP$ such that $\|p_n\chi_x\|<2^{-n}$ for all $x\in X$. Notice that for each finite $A\subseteq X$, $\lim_{n\to \infty}p_n\chi_A=0$. In fact, 
for each $n$, $\|p_n\chi_A\|\leq \sum_{x\in A}\|p_n\chi_x\|\leq 2^{-n}|A|$, and as $A$ is fixed, the conclusion follows.

As each $p_n$ is finite dimensional, we can go to a subsequence and find (modulo reindexing) a disjoint sequence $(X_n)_n$ of finite subsets of $X$ such that 
\begin{equation}\label{Eq.smallpn}
\|p_n-\chi_{X_n}p_n\chi_{X_n}\|<2^{-n-1}.
\end{equation}
for all $n\in\N$. Since $\mathcal P$ is equi-approximable and each $\chi_{X_n}$ has propagation zero, the sequence $(\chi_{X_n}p_n\chi_{X_n})_n$ is equi-approximable. Moreover, as $(X_n)_n$ are disjoint, this implies that
\[
p'=\SOTh\sum_n \chi_{X_n}p_n\chi_{X_n}\in\cstu(X).
\]
Notice that $p'$ is a ghost. 
Each one of the operators $a_n=\chi_{X_n} p_n \chi_{X_n}$ is positive, and by \eqref{Eq.smallpn} these operators satisfy $\lim_{n\to \infty}\|a_n-a_n^2\|=0$. By standard continuous functional calculus argument using \eqref{Eq.smallpn}, for each $n\in\N$ we can find a projection $q_n\in\mathcal B(\ell_2(X_n))$ such that 
\[
\norm{q_n- \chi_{X_n}p_n\chi_{X_n}}<2^{-n+1}.
\]
In the situation of \eqref{LemmaGeoCondForEquiApproxAreBoundedBelow2}, $p_n$ has rank one, and therefore $\chi_{X_n} p_n \chi_{X_n}$ has rank one, and the same applies to $q_n$.

As $p_n$ is nonzero, $q_n$ is nonzero, and therefore 
\[
q=\SOTh\sum_nq_n
\] is a noncompact block projection. Moreover \[
p'-q=\sum_n (\chi_{X_n}p_n\chi_{X_n}-q_n)
\]
in norm, so, as each $\chi_{X_n}p_n\chi_{X_n}-q_n$ is compact, $p'-q$ is also compact. As $p'$ is a ghost, so is $q$. This contradicts the fact that all ghost block projections in $\cstu(X)$ are compact.
\end{proof} 


\begin{proof}
[Proof of Theorem \ref{ThmBlockProjIFFONLproj}] 
\eqref{ItemThmBlockProjIFFONLproj2.5}$\Rightarrow$\eqref{ItemThmBlockProjIFFONLproj3}: 
If all ghost block projections in $\cstu(X)$ are compact, then Lemma \ref{LemmaGeoCondForEquiApproxAreBoundedBelow} implies the assumptions of Lemma \ref{LemmaSortONLForEquiApproxFamOfProj}, and therefore $X$ has ONL for finite rank projections. 

Suppose now that $(X,\cE)$ is metrizable, say $\cE=\cE_d$ for some metric $d$ on~$X$. 

Clearly \eqref{ItemThmBlockProjIFFONLproj2.5}$\Rightarrow$\eqref{ItemThmBlockProjIFFONLproj2}. Conversely, suppose \eqref{ItemThmBlockProjIFFONLproj2} holds and \eqref{ItemThmBlockProjIFFONLproj2.5} fails, and let $\sum_n p_n$ be a noncompact ghost block projection in $X$. Then all $p_n$ are nonzero, and we can find an infinite $M\subseteq \bbN$ such that $\sum_{n\in M} p_n$ belongs to a sparse subspace $X’$ of $X$, and it is therefore compact; contradiction. 

Since \eqref{ItemThmBlockProjIFFONLproj1}$\Rightarrow$\eqref{ItemThmBlockProjIFFONLproj2} is trivial, we are left to show that \eqref{ItemThmBlockProjIFFONLproj3}$\Rightarrow$\eqref{ItemThmBlockProjIFFONLproj2.5} and \eqref{ItemThmBlockProjIFFONLproj2}$\Rightarrow$\eqref{ItemThmBlockProjIFFONLproj1}. 
 
\eqref{ItemThmBlockProjIFFONLproj3}$\Rightarrow$\eqref{ItemThmBlockProjIFFONLproj2.5}: Fix a noncompact block projection with respect to a disjoint sequence $(X_n)_n$ of finite subsets of $X$, say $p\in \cstu(X)$. For each $n\in\N$, let $p_n=\chi_{X_n}p\chi_{X_n}$; so $p=\SOTh\sum_np_n$. 

\begin{claim}\label{Claim3.7}
We have that $\SOTh\sum_{n\in M}p_n\in \cstu(X)$ for all $M\subseteq \N$.
\end{claim}

\begin{proof}
For $M\subseteq\N$, let $X_M=\bigcup_{n\in M} X_n$. Then $\SOTh\sum_{n\in M}p_n=\chi_{X_M}p\chi_{X_M}$. Since $p\in\cstu(X)$, we have the thesis.
\end{proof}

By Claim~\ref{Claim3.7} and Lemma \ref{LemmaUnifApprox}, $(p_n)_n$ is equi-approximable. Hence, by our hypothesis, there is $r>0$ such that every $p_n$ is $(1/2,r)$-normed. For each $n\in\N$, let $A_n\subseteq X$ be a such that $\diam(A_n)\leq r$ and $\|p_n\chi_{A_n}\|\geq 1/2$. As $X$ is u.l.f., $k=\sup_{n\in\N}|A_n|<\infty$. Suppose for a contradiction that $p$ is a ghost. Then there is $n\in\N$ such that $\|p_n\delta_x\|<1/(2k)$ for all $x\in X$. So, \[\|p_n\chi_{A_n}\|\leq\sum_{x\in A_n}\|p_n\delta_x\|< 1/2;\]contradiction.

\eqref{ItemThmBlockProjIFFONLproj2}$\Rightarrow$\eqref{ItemThmBlockProjIFFONLproj1}: Let $X'=\bigsqcup_nX_n$ be a sparse subspace of $X$ and let $p\in \cstu(X')$ be a ghost projection. Since $X'$ is sparse, if $a$ is a finite propagation operator in $\cstu(X')$ then there is $n_0\in\N$ such that \[a-\SOTh\sum_{m\geq n_0}\chi_{X_m}a\chi_{X_m}\] is compact. This shows that passing to the Calkin algebra, we have
\[
\cstu(X')/\mathcal K(\ell_2(X'))\subseteq\Big(\prod_{n} \cK(\ell_2(X_n))\Big)/ \cK(\ell_2(X')).
\]
Therefore, as 
\[
\Big(\prod_{n} \cK(\ell_2(X_n))\Big)/ \cK(\ell_2(X'))=\prod_{n} \cK(\ell_2(X_n))/\bigoplus\cK(\ell_2(X_n))
\]
 and since all projections in \[\prod_n\mathcal \cK(\ell_2(X_n))/\bigoplus \cK(\ell_2(X_n))\] lift to a projection in $\prod_n \cK(\ell_2(X_n))$ (see \cite[Lemma~3.1.13]{Fa:STCstar}), we can pick a projection $p'\in \prod_n \mathcal K(\ell_2(X_n))$ such that $p-p'$ is compact. As $\mathcal K(\ell_2(X'))\subseteq\cstu(X')$ and $p\in \cstu(X')$, we have that $p'\in\cst(X')$. As $p$ is a ghost, so is $p'$. Therefore, $p'$ is a ghost block projection in $\cstu(X')$ and, by our hypothesis, $p'$ must be compact. This gives us that $p$ is compact. Since $p$ was an arbitrary ghost projection in $\cstu(X’)$ for an arbitrary sparse subspace $X’$ of $X$, we are done.
\end{proof}

\begin{proof}[Proof of Theorem \ref{ThmBlockRank1ProjIFFONLproj}]
The proof is analogous to that of Theorem \ref{ThmBlockProjIFFONLproj}, with `finite rank projections' replaced with `rank one projections' everywhere and (4) replaced with (3). 
\end{proof}

\section{Applications I: Rigidity and Cartan masas}\label{SectionCoSep}

In this section we use Theorem \ref{ThmBlockRank1ProjIFFONLprojINTRO} to prove Theorems \ref{Thm.PreservationMetric}, \ref {Thm.CartanCosep}, and \ref{Thm.GhostsAndCosep}. The conclusions of these theorems were previously known to hold only in presence of ONL. 

In order to prove Theorem \ref{Thm.PreservationMetric}, we will need some technical results already proven in the literature. For the reader's convenience, we state those results below. Given coarse spaces $(X,\cE)$ and $(Y,\cF)$, an isomorphism $\Phi\colon\cstu(X)\to \cstu(Y)$, $x\in X$, $y\in Y$, and $\eta>0$, we let
\begin{itemize}
\item$ X_{y,\eta} \coloneqq \{z\in X\mid \|\Phi^{-1}(\chi_y)\delta_z\|\geq \eta\}$, and 
\item $Y_{x,\eta} \coloneqq \{z\in Y\mid \|\Phi(\chi_x)\delta_z\|\geq \eta\}$.
\end{itemize}

We isolate two results from \cite{BaudierBragaFarahKhukhroVignatiWillett2021uRaRig} and \cite{BragaFarahVignati2020AnnInstFour}.
 
\begin{lemma}[{\cite[Corollary 5.3]{BaudierBragaFarahKhukhroVignatiWillett2021uRaRig}}]
Let $(X,\cE)$ and $(Y,\cF)$ be u.l.f.\ coarse spaces and $\Phi\colon \cstu(X)\to \cstu(Y)$ be an isomorphism. If $X$ is metrizable, then for all $\eps>0$ there is $\eta>0$ such that $\|\Phi (\chi_x)(1_{\ell_2(Y)}- \chi_{Y_{x,\eta}})\|\leq \eps$ for all $x\in X$.\label{LemmaCor.Expanding}
\end{lemma} 

\begin{lemma}[{\cite[Lemma 4.7]{BragaFarahVignati2020AnnInstFour}}]
Let $(X,\mathcal{E})$ and $(Y,\mathcal{F})$ be u.l.f.\ coarse spaces, $\Phi\colon\cstu(X)\to \cstu(Y)$ be an isomorphism, and let $f\colon X\to Y$ be such that $\inf_{x\in X}\|\Phi(\chi_x)\delta_{f(x)}\|>0$. Then the following holds:
\begin{enumerate} 
\item If for all $\eps>0$ there is $\eta>0$ such that 
\[
\|\Phi (\chi_x)(1_{\ell_2(Y)}- \chi_{Y_{x,\eta}})\|\leq \eps,
\] 
for all $x\in X$, then $f$ is expanding.
\item If for all $\eps>0$ there is $\eta>0$ such that 
\[
\|\Phi^{-1}(\chi_{f(x)})(1_{\ell_2(X)}- \chi_{X_{f(x),\eta}})\|\leq \eps,
\]
for all $x\in X$, then $f$ is coarse.\qed
\end{enumerate} \label{LemmaTheMapIsCoarse.GenCoarseSp}
\end{lemma} 
 
\begin{proof} [Proof of Theorem \ref{Thm.PreservationMetric}]Suppose $(X,\cE)$ and $(Y,\cF)$ are u.l.f.\ coarse spaces and $\Phi\colon \cstu(X)\to\cstu(Y)$ is an isomorphism. Also suppose $X$ is metrizable and $\cstu(X)$ has no block-rank-one ghost projections. We need to prove that $(X,\cE)$ and $(Y,\cF)$ are coarsely equivalent. 

Since $X$ is metrizable and $\cstu(X)$ has no block-rank-one ghost projections, replacing $\delta$ by a smaller positive number if necessary, Lemmas~\ref{LemmaGeoCondForEquiApproxAreBoundedBelow} and \ref{LemmaUnifApprox} together imply that there is $g\colon Y\to X$ such that $\|\Phi^{-1}(\chi_y)\delta_{g(y)}\|>\delta$ for all $y\in Y$.\footnote{Alternatively, this follows from \cite[Corollary 3.3]{LiSpakulaZhang2020}.} Let us show $g$ is a coarse embedding.

By Lemma \ref{LemmaCor.Expanding}, for all $\eps>0$ there is $\eta>0$ such that 
\[
\|\Phi(\chi_x)(1_{\ell_2(Y)}-\chi_{Y_{x,\eta}})\|\leq \eps.
\]
Therefore, by Lemma \ref{LemmaTheMapIsCoarse.GenCoarseSp}, $g$ is coarse. We now show that $g$ is also expanding. 

\begin{claim}\label{claim:bigonsmall}
For all $\eps>0$ there is $\eta>0$ such that $\|\Phi^{-1}(\chi_y)\chi_{X_{y,\eta}}\|\geq 1-\eps$ for all $y\in Y$.
\end{claim}

\begin{proof}
As $X$ is a metrizable space, it follows from Lemma \ref{LemmaUnifApprox} that the indexed family $(\Phi^{-1}(\chi_y))_{y\in Y}$ is equi-approximable. Therefore, as $\cstu(X)$ has no ghost block-rank-one projection, Theorem \ref{ThmBlockRank1ProjIFFONLproj} gives that ONL holds for $(\Phi^{-1}(\chi_y))_{y\in Y}$. The remaining part of the proof now closely follows the proof of \cite[Lem\-ma 6.7]{WhiteWillett2017}. Alternatively, and using a terminology closer to that of the present paper, the proof can be completed by using \cite[Lemma 7.4]{BragaFarahVignati2018}. Indeed, although \cite[Lemma 7.4]{BragaFarahVignati2018} assumes the metric spaces have ONL, the only thing necessary for its argument to hold is that ONL holds for $(\Phi^{-1}(\chi_y))_{y\in Y}$. So, we are done. 
\end{proof}

As each $\Phi^{-1}(\chi_y)$ has rank 1,
\[
\|\Phi^{-1}(\chi_y)(1_{\ell_2(X)}-\chi_{X_{y,\eta}})\|^2=\|\Phi^{-1}(\chi_y)1_{\ell_2(X)}\|^2-\|\Phi^{-1}(\chi_y)\chi_{X_{y,\eta}}\|^2.
\]
By Claim~\ref{claim:bigonsmall}, the second term on the right-hand side can be made greater than $1-\varepsilon$ by choosing a small enough $\eta>0$. Since $\varepsilon>0$ was arbitrary, Lemma \ref{LemmaTheMapIsCoarse.GenCoarseSp} now implies that~$g$ is expanding.

The result now follows immediately from \cite[Theorem 1.2]{BaudierBragaFarahKhukhroVignatiWillett2021uRaRig}. Indeed, as $(Y,\cF)$ coarsely embeds into $(X,\cE)$ and $(X,\cE)$ is metrizable, $(Y,\cF)$ is also metrizable. Therefore, \cite[Theorem 1.2]{BaudierBragaFarahKhukhroVignatiWillett2021uRaRig}, $(X,\cE)$ and $(Y,\cF)$ are actually coarsely equivalent. 

It remains to prove that $g$ is a coarse equivalence. By Corollary \ref{CorRigPaper}, there are $\delta>0$ and $f\colon X\to Y$ such that $\|\Phi (\chi_x)\delta_{f(x)}\|> \delta$ for all $x\in X$. By Lemma \ref{LemmaMapWitRig}, $f$ is a coarse embedding. We need to verify that $g\circ f$ and $f\circ g$ are close to $\mathrm{Id}_X$ and $\mathrm{Id}_Y$, respectively. Let us first show that $g\circ f$ is close to $\mathrm{Id}_X$. As $X$ is metrizable, Lemma \ref{LemmaUnifApprox} gives $E\in \cE$ such that every $\Phi^{-1}(\chi_y)$ is $\delta^2$-$E$-approximable. Therefore, since 
\begin{align*}
\|\chi_{g(f(x))}\Phi^{-1}(\chi_{f(x)})\chi_x\|&\leq \|\Phi^{-1}(\chi_{f(x)})\chi_{g(f(x))}\|\|\Phi^{-1}(\chi_{f(x)})\chi_x\|\\
&=\|\Phi^{-1}(\chi_{f(x)}) \chi_{g(f(x))}\|\|\Phi(\chi_x)\chi_{f(x)}\|<\delta^2
\end{align*}
for all $x\in X$, we must have that $(x,g(f(x)))\in E$ for all $x\in X$. 

We now show $f\circ g$ is close to $\mathrm{Id}_Y$. This follows from the following general fact in coarse geometry: as $g$ is expanding and $g\circ f$ is close to $\mathrm{Id}_X$, then $f\circ g$ is close to $\mathrm{Id}_Y$. For completeness, we prove this simple fact. As $g\circ f$ is close to $\mathrm{Id}_X$, fix $E\in \cE$ such that $(x,g(f(x)))\in E$ for all $x\in X$. As $g$ is expanding, there is $F\in \cF$ such that $(y,z)\not\in F$ implies $(g(y),g(x))\not\in E$. In particular, if $(y,f(g(y)))\not\in F$ for some $y\in Y$, then $(g(y),g(f(g(y))))\not\in E$. By our choice of $E$, this cannot happen. Therefore, $(y,f(g(y)))\in F$ for all $y\in F$.
\end{proof}
 
\begin{proof}[Proof of Theorem \ref{Thm.CartanCosep}]Suppose that $X$ is a u.l.f.\ metric space, $\cstu(X)$ has no block-rank-one ghost projections, and $B\subseteq \cstu(X)$ is a Cartan masa (Definition~\ref{Def.Cartan}) isomorphic to $\ell_\infty(\N)$. We need to prove that $B$ is co-separable in $\cstu(X)$.

By \cite[Theorem 4.17]{WhiteWillett2017}, there is a u.l.f.\ coarse space $(Y,\cF)$ and an isomorphism $\Phi\colon \cstu(X)\to \cstu(Y)$ such that $\Phi[B]=\ell_\infty(Y)$. By Theorem \ref{Thm.PreservationMetric}, $Y$ is metrizable. Hence by \cite[Lemma 4.19]{WhiteWillett2017}, $\ell_\infty(Y)$ is co-separable in $\cstu(Y)$. This in turns implies that $B$ is co-separable in $\cstu(X)$. 
\end{proof} 

The next lemma is an easier version of part \eqref{LemmaGeoCondForEquiApproxAreBoundedBelow1} of Lemma \ref{LemmaGeoCondForEquiApproxAreBoundedBelow} and several small variations of it have already been were obtained in the literature, probably starting with \cite[Theorem 6.2]{BragaFarah2018Trans}. For this reason, we omit its proof.

\begin{lemma}\label{LemmaGGGG}
Let $(X,\cE)$ be a u.l.f.\ coarse space and let $(p_n)_n$ be a family of finite-rank projections in $\cstu(X)$ such that, for all infinite $M\subseteq \N$, $\SOTh\sum_{n\in M}p_n$ is a non-ghost projection in $\cstu(X)$. Then 
\[
\inf_{n\in\N}\sup_{x\in X}\|p_n\chi_x\|>0.\eqno\qed
\]
\end{lemma} 
 
\begin{proof}[Proof of Theorem \ref{Thm.GhostsAndCosep}] Suppose $(X,d)$ is a u.l.f.\ metric space and $B\subseteq \cstu(X)$ is a Cartan masa isomorphic to $\ell_\infty(\N)$. 

\eqref{Thm.GhostsAndCosep.Item1}$\Rightarrow$\eqref{Thm.GhostsAndCosep.Item2}: If all ghost operators in $B$ are compact, then all ghost projections in $B$ are compact. 
 
\eqref{Thm.GhostsAndCosep.Item2}$\Rightarrow$\eqref{Thm.GhostsAndCosep.Item3}: Suppose that all ghost projections in $B$ are compact. A proof that $B$ is co-separable in $\cstu(X)$ is analogous to that of Theorem \ref{Thm.PreservationMetric}. Precisely, by \cite[Theorem 4.17]{WhiteWillett2017}, there are a u.l.f.\ coarse space $(Y,\cF)$ and an isomorphism $\Phi\colon \cstu(X)\to \cstu(Y)$ such that $\Phi(B)=\ell_\infty(Y)$. Then, by Lemma \ref{LemmaGGGG}, there are $\delta>0$ and a map $g\colon Y\to X$ such that $\|\Phi^{-1}(\chi_y)\delta_{g(y)}\|>\delta$ for all $x\in X$ and $y\in Y$. Proceeding as in the proof of Theorem \ref{Thm.PreservationMetric}, we conclude that $g$ is a coarse embedding. In particular, $(Y,\cF)$ is metrizable. By \cite[Lemma 4.19]{WhiteWillett2017}, this implies that $\ell_\infty(Y)$ is co-separable in $\cstu(Y)$. Hence, $B$ is co-separable in $\cstu(X)$. 
 
\eqref{Thm.GhostsAndCosep.Item3}$\Rightarrow$\eqref{Thm.GhostsAndCosep.Item1}: Suppose $B$ is co-separable in $\cstu(X)$. By \cite[Theorem B]{WhiteWillett2017} there are a u.l.f.\ metric space $(Y,\partial)$ and an isomorphism $\Phi\colon \cstu(X)\to \cstu(Y)$ such that $\Phi(B)=\ell_\infty(Y)$. By \cite[Theorem 1.11]{BaudierBragaFarahKhukhroVignatiWillett2021uRaRig}, any isomorphism between uniform Roe algebras of metric spaces must send ghosts to ghosts. As all ghosts in $\ell_\infty(Y)$ are compact, the same must happen in $B$. \end{proof}

\section{Applications II: Embeddings}\label{SectionEmb}

The study of embeddings between uniform Roe algebras was initiated in \cite{BragaFarahVignati2019Comm}. It is known that the existence of an embedding $\cstu(X)\hookrightarrow \cstu(Y)$ does not necessarily imply that $X$ coarsely embeds into $Y$. Indeed, an embedding $\cstu(X)\hookrightarrow \cstu(Y)$ exists if there is an injective coarse map $X\to Y$. Hence, letting $f\colon \Z\to \N$ be such that $f(0)=0$, and $f(n)=2n$ and $f(-n)=2n-1$ for all $n\in\N$, we conclude that $\cstu(\Z)$ embeds into $\cstu(\N)$ but $\Z$ clearly does not coarsely embed into $\N$ (cf. \cite[\S2.4]{BragaFarahVignati2019Comm}).

On the other hand, the existence of an embedding $\cstu(X)\hookrightarrow \cstu(Y)$ satisfying some extra conditions (e.g., rank preservation, compact preservation, hereditary range, etc) is often enough to give us some information on how to map $X$ into $Y$. In this section, we study how the new rigidity techniques introduced in \cite{BaudierBragaFarahKhukhroVignatiWillett2021uRaRig} apply to embeddings. More precisely, we investigate when the analog of Corollary \ref{CorRigPaper} holds for embeddings in place of isomorphisms. 

We start with a simple, but hopefully somewhat enlightening, observation. 

\begin{proposition}\label{ThmEmbCanonicalMASAToMASA}
Let $(X,\cE)$ and $(Y,\cF)$ be u.l.f.\ coarse spaces. The following are equivalent for an embedding $\Phi\colon \cstu(X)\to \cstu(Y)$:

\begin{enumerate}
\item\label{Item1.ThmEmbCanonicalMASAToMASA} $\Phi$ is unital, rank-preserving, and strongly continuous.
\item \label{Item1+.ThmEmbCanonicalMASAToMASA} $\Phi$ is implemented by a unitary $u\colon \ell_2(X)\to \ell_2(Y)$. 
\item \label{Item2-.ThmEmbCanonicalMASAToMASA} $\Phi(\ell_\infty(X))$ is a masa of $\cB(\ell_2(Y))$.
\item \label{Item2.ThmEmbCanonicalMASAToMASA} $\Phi(\ell_\infty(X))$ is a masa of $\cstu(Y)$.
\end{enumerate}
\end{proposition}

We do not know whether the conditions in Proposition~\ref{ThmEmbCanonicalMASAToMASA} suffice to guarantee that $\inf_{x\in X}\sup_{y\in Y}\|\Phi(\chi_x)\delta_{y}\|>0$. Its proof of uses the following minor strengthening of \cite[Lemma 2.3]{WhiteWillett2017}, where the same conclusion was obtained under the stronger assumption that $A$ includes all compact operators.

\begin{lemma}\label{L.WW}
Suppose that $A$ is a concrete subalgebra of $\cB(H)$ for some Hilbert space that includes a masa of $\cB(H)$ isomorphic to $\ell_\infty(Z)$ for some set $Z$. If~$B$ is a masa in $A$, then every minimal projection in $B$ has rank 1. 
\end{lemma}

\begin{proof}Fix a masa $B$ in $A$ and a minimal projection $p$ in $B$. 
The map $a\mapsto pap$ is a SOT-continuous, hence $\cP=\{pqp\mid q\in \ell_\infty(Z)$ has finite rank$\}$ is a directed net of positive contractions included in $A$ with $p$ as its supremum. By SOT-continuity we can fix a rank 1 projection $q\in \ell_\infty(Z)$ such that $pqp\neq 0$. By the minimality of $p$, $pqp$ is a scalar multiple of $p$ and therefore $p$ has rank 1. 
\end{proof}

\begin{proof}[Proof of Proposition~\ref{ThmEmbCanonicalMASAToMASA}]
\eqref{Item1.ThmEmbCanonicalMASAToMASA} implies \eqref{Item1+.ThmEmbCanonicalMASAToMASA}: 
This is similar to the proof of \cite[Lemma 3.1]{BragaFarah2018Trans}. Suppose that $\Phi$ is as in \eqref{Item1.ThmEmbCanonicalMASAToMASA}. Write $X$ as a discrete sum of its connected components, $X=\bigsqcup_i X_i$. It suffices to prove that the restriction of $\Phi$ to each $X_i$ is implemented by a unitary, so we may assume that $X$ is connected.  By assumptiuon $\Phi$ sends rank 1 projections to rank 1 projections; as moreover it is unital and SOT-continuous, its range contains a maximal orthogonal set of rank one projections.  As $\Phi$ is a $*$-homomorphism, it preserves Murray-von Neumann equivalence of projections, and so the range of $\Phi$ includes the ideal of compact operator on $\ell_2(Y)$.  The conclusion follows as in \cite[Lemma 3.1]{BragaFarah2018Trans}.

Clearly \eqref{Item1+.ThmEmbCanonicalMASAToMASA} implies \eqref{Item2-.ThmEmbCanonicalMASAToMASA} and 
\eqref{Item2-.ThmEmbCanonicalMASAToMASA} implies \eqref{Item2.ThmEmbCanonicalMASAToMASA}. 

Suppose that $\Phi$ satisfies \eqref{Item2.ThmEmbCanonicalMASAToMASA}. In particular, $\Phi(1_{\ell_2(X)})=1_{\ell_2(Y)}$, so $\Phi$ is unital. By Lemma~\ref{L.WW} applied to $A=\cstu(Y)$, every minimal projection in $\Phi[\ell_\infty(X)]$ has rank 1. 
 Therefore, each $\Phi(e_{xz})$ also has rank-one and it follows that 
\[
\Phi\restriction \bigcup_{F\subseteq X,|F|<\infty}\cB(\ell_2(F))
\]
is rank-preserving. As each finite-rank operator operator can be norm approximated by operators in $ \bigcup_{F\subseteq X,|F|<\infty}\cB(\ell_2(F))$ with the same rank, it easily follows that $\Phi$ is rank-preserving.

It remains to show that $\Phi$ is strongly continuous. For that, notice that as $\Phi\colon \ell_\infty(X)\to \Phi(\ell_\infty(X))$ is an isomorphism, as each $\Phi(\chi_x)$ has rank-one, and as $\Phi(\ell_\infty(X))$ is a masa in $\cstu(Y)$, we have that that $(\Phi(\chi_x))_{x\in X}$ is a maximal family of rank-one projections. Therefore, 
\[
\SOTh\sum_{x\in X}\Phi(\chi_x)=1_{\ell_2(Y)}=\Phi(1_{\ell_2(X)}).
\]
Since $\Phi$ is strongly continuous if and only if $\Phi(1_{\ell_2(X)})=\SOTh\sum_{x\in X}\Phi(\chi_x)$ (see, for instance, \cite[Corollary 5.2]{Braga2020CoarseQuotient}\footnote{\cite[Corollary 5.2]{Braga2020CoarseQuotient} is based on \cite[Theorem 4.3]{BragaFarahVignati2019Comm}. Even though metrisability of $X$ and $Y$ is stated as a hypothesis, it is not used in the proof of \cite[Theorem~4.3]{BragaFarahVignati2019Comm}, which holds for u.l.f. coarse spaces in full generality. See also \cite[Theorem~6.1]{BaudierBragaFarahVignatiWillett2022vNA}.}, this shows that $\Phi$ satisfies \eqref{Item1.ThmEmbCanonicalMASAToMASA}. 
\end{proof}

We continue with a positive result.
 
\begin{theorem}\label{ThmEmb1}
Let $(X,\cE)$ and $(Y,\cF)$ be u.l.f.\ coarse spaces. If $(X,\cE)$ is metrizable, $\Phi\colon \cstu(X)\to \cstu(Y)$ is an embedding, and at least one of the following conditions applies
\begin{enumerate}
\item\label{CorRigPaperEmb1} $(Y,\cF)$ is metrizable and $\Phi(\ell_\infty(X))$ is a co-separable Cartan masa of~$\cstu(Y)$, or
\item \label{CorRigPaperEmb2} $\ell_\infty(Y)\subseteq \Phi(\cstu(X))$, 
\end{enumerate}
then there is a coarse and uniformly finite-to-one $f\colon X\to Y$ such that 
\[
\inf_{x\in X}\|\Phi(\chi_x)\delta_{f(x)}\|>0.
\]
\end{theorem}

\begin{proof}
Suppose $(Y,\cE)$ is metrizable and that $\Phi(\ell_\infty(X))$ is a co-separable Cartan masa of $\cstu(Y)$. By Theorem \ref{Thm.GhostsAndCosep}, all ghosts in $\Phi(\ell_\infty(X))$ are compact. Therefore, Lemma \ref{LemmaGGGG} gives a map $f\colon X\to Y$ such that 
\[
\inf_{x\in X}\|\Phi(\chi_x)\delta_{f(x)}\|>0.
\]
 As $(Y,\cF)$ is metrizable, Lemma \ref{LemmaMapWitRig} implies $f$ is coarse and uniformly finite-to-one.

Suppose now that $\ell_\infty(Y)\subseteq \Phi(\cstu(X))$. Then, for each $M\subseteq Y$, we have 
\[
\Phi^{-1}\Big(\SOTh\sum_{y\in M}\chi_y\Big)=\SOTh\sum_{y\in M}\Phi^{-1}(\chi_y).
\]
Indeed, fix $M\subseteq Y$ and let $p$ be a finite rank projection below the projection $\Phi^{-1} (\SOTh\sum_{y\in M}\chi_y )$ and orthogonal to all $(\Phi^{-1}(\chi_y))_{y\in M}$. As $X$ is metrizable, $\cstu(X)$ contains the compacts; so, $p\in \cstu(X)$. Therefore, $\Phi(p)$ is a projection below $\SOTh\sum_{y\in M}\chi_y$ which is orthogonal to all $(\chi_y)_{y\in M}$, i.e., $\Phi(p)=0$. As $\Phi$ is injective, $p=0$.
 
As $\ell_\infty(Y)\subseteq \Phi(\cstu(X))$, the previous paragraph implies that the projections $(\Phi^{-1}(\chi_y))_{y\in Y}$ satisfy
\begin{enumerate}
\item $\SOTh\sum_{y\in M}\Phi^{-1}(\chi_y)\in \cstu(X)$ for all $M\subseteq Y$, and
\item $\SOTh\sum_{y\in Y}\Phi^{-1}(\chi_y)=1_{\ell_2(X)}$. 
\end{enumerate}
Therefore, by Corollary \ref{LemmaCorRigPaper}, there are $\delta>0$ and $f\colon X\to Y$ such that $\|\Phi^{-1}(\chi_{f(x)})\delta_{x}\|>\delta$ for all $x\in X$. Hence, 
\[
\|\Phi(\chi_x)\delta_{f(x)}\|=\|\Phi^{-1}(\chi_{f(x)})\delta_{x}\|>\delta
\]
for all $x\in X$. Again, Lemma \ref{LemmaMapWitRig} implies that $f$ is coarse and uniformly finite-to-one.
\end{proof}

\begin{remark}
Notice that the existence of a uniformly finite-to-one coarse map $X\to Y$ between u.l.f.\ coarse spaces is enough so that the geometry of $X$ is highly controlled by the one of $Y$. For instance, if such map exists, (1) the asymptotic dimension of $X$ is bounded by the one of $Y$, (2) if $Y$ has property A, so does $X$, and (3) if $Y$ has finite decomposition complexity, then so does $X$ (see \cite[Corollary 1.3]{BragaFarahVignati2019Comm}). 
\end{remark}
 
The reader familiar with the theory of embeddings of uniform Roe algebras may have noticed that the conditions in Theorem \ref{ThmEmbCanonicalMASAToMASA} radically differ from the ones previously considered in this context. Precisely, the results in the literature (see \cite[Theorems 1.4 and 5.4]{BragaFarahVignati2019Comm}) usually require the target space $Y$ to satisfy some geometric condition and the embedding $\Phi\colon \cstu(X)\to\cstu(Y)$ to either (1) be rank-preserving or (2) have its image to be a hereditary subalgebra of $\cstu(Y)$. 

It is then natural to wonder if the techniques used to prove Corollary \ref{CorRigPaper} can be applied to obtain a version of this corollary for embeddings. The following shows that this is not possible. 

\begin{proposition}\label{PropMethodFailsForEmbHer}
Let $X=\bigsqcup_n{\{x_n\}}$ be the coarse disjoint union of singletons.\footnote{For instance, the reader can have $X=\{n^2\mid n\in\N\}$ with the metric inherited from $\bbN$ in mind.} Given a u.l.f.\ metric space $Y$, the following are equivalent.
\begin{enumerate}
\item\label{Item1.PropMethodFailsForEmbHer} There is a block-rank-one ghost projection in $ \cstu(Y)$.
\item \label{Item2.PropMethodFailsForEmbHer} There is an embedding $\Phi\colon \cstu(X)\to\cstu(Y)$ onto a hereditary subalgebra of $\cstu(Y)$ such that $\inf_{x\in X}\sup_{y\in Y}\|\Phi(\chi_x)\delta_{y}\|=0$.
\item \label{Item3.PropMethodFailsForEmbHer} There is a unital rank-preserving embedding $\Phi\colon \cstu(X)\to\cstu(Y)$ such that $\inf_{x\in X}\sup_{y\in Y}\|\Phi(\chi_x)\delta_{y}\|=0$.
\end{enumerate}
\end{proposition}

\begin{proof}
Since $X$ is the coarse disjoint union of singletons, finite propagation elements of $C^*_u(X)$ with zero diagonal are finite rank.  Hence 
\[
\cstu(X)=\ell_\infty(X)+\cK(\ell_2(X)).
\] 
Therefore, in order to define a $^*$-homomorphism $\Phi\colon\cstu(X)\to \cstu(Y)$, it is enough to define $\Phi$ on $\ell_\infty(X)$ and on each $e_{x_nx_m}$, and then extend it linearly and continuously to the whole $\cstu(X)$, where $e_{xz}$ is the rank $1$ partial isometry mapping $\delta_x$ to $\delta_z$, for $x,z\in X$.

\eqref{Item1.PropMethodFailsForEmbHer}$\Rightarrow$\eqref{Item2.PropMethodFailsForEmbHer}: Let $p\in \cstu(Y)$ be a block-rank-one ghost projection and let $(Y_n)_n$ be a disjoint sequence of finite subsets of $Y$ such that \[p=\SOTh\sum_{n\in\N}\chi_{Y_n}p\chi_{Y_n}.\] Replacing $(Y_n)_n$ by a subsequence if necessary, we can assume that $Y'=\bigsqcup_nY_n$ is sparse. For each $n\in\N$, let $p_n=\chi_{Y_n}p\chi_{Y_n}$ and pick a normalized $\xi_n\in \ell_2(X)$ such that $p_n=\langle \cdot ,\xi_n\rangle \xi_n$. 

Define an embedding $\Phi\colon \cstu(X)\to \cstu(Y)$ by $\Phi(\sum_{n\in M}\chi_{x_n})=\sum_{n\in M}p_n$ for all $M\subseteq \N$, and $\Phi(e_{x_nx_m})=\langle \cdot, \xi_m\rangle \xi_n$ for all $n,m\in\N$. As $\Phi(\chi_{x_n})=p_n$ for all $n\in\N$ and as $p$ is a ghost, then $ \inf_{x\in X}\sup_{y\in Y}\|\Phi(\chi_x)\delta_{y}\|=0$.
 
It remains to notice that $\Phi(\cstu(X))$ is a hereditary subalgebra of $\cstu(Y)$. 
 As $\cstu(Y')$ is a hereditary subalgebra of $\cstu(Y)$ and as $p=\Phi(1_{\ell_2(X)})$, it is sufficient to show that $p\cstu(Y') p\subseteq \Phi(\cstu(X))$. For that, it is enough to notice that $p\cstu[Y'] p\subseteq \Phi(\cstu(X))$. Since $Y'$ is sparse and each $p_n$ has rank~1, this is straightforward and we leave the details to the reader. 

\eqref{Item1.PropMethodFailsForEmbHer}$\Rightarrow$\eqref{Item3.PropMethodFailsForEmbHer}: Let $Y'=\bigsqcup_n Y_n$ and $p=\SOTh\sum_n p_n$ be as in \eqref{Item1.PropMethodFailsForEmbHer}.  Note that as~$p$ is a ghost and $1$ is not, $1-p$ has infinite rank.   Let $\Phi\colon \cstu(X)\to \cstu(Y)$ be a $*$-homomorphism constructed as in the proof of \eqref{Item1.PropMethodFailsForEmbHer}$\Rightarrow$\eqref{Item2.PropMethodFailsForEmbHer}, so in particular, $\Phi(1)=p$.  Fix a nonprincipal ultrafilter $\cU$ on $\N$ and let $\Psi\colon \cstu(X)\to \cstu(Y)$ be a $^*$-homomorphism determined by $\Psi(\cK(\ell_2(X)))=\{0\}$ and 
\[
\Psi(\chi_{A})=\left\{\begin{array}{ll}
1-p,& \text{ if } A \in \cU,\\
0,& \text{ if }A\not\in \cU.
\end{array}\right.
\]
As $\Phi(a)\Psi(b)=\Psi(b)\Phi(a)=0$ for all $a,b\in \cstu(X)$, the map $\Theta=\Phi+\Psi$ is a $^*$-homomorphism, hence an embedding.\footnote{Notice that $\Theta$ is not strongly continuous; see \cite[Theorem~6.1]{BaudierBragaFarahVignatiWillett2022vNA} for the proper context.} Moreover, it is clear that $\Theta$ is unital, rank preserving, and that $\inf_{x\in X}\sup_{y\in Y}\|\Theta(\chi_x)\delta_{y}\|=0$.

\eqref{Item2.PropMethodFailsForEmbHer}$\Rightarrow$\eqref{Item1.PropMethodFailsForEmbHer} and \eqref{Item3.PropMethodFailsForEmbHer}$\Rightarrow$\eqref{Item1.PropMethodFailsForEmbHer}: Suppose \eqref{Item1.PropMethodFailsForEmbHer} fails and let us show that both \eqref{Item2.PropMethodFailsForEmbHer} and \eqref{Item3.PropMethodFailsForEmbHer} must fail. For that, consider an embedding $\Phi\colon \cstu(X)\to\cstu(Y)$. If $\Phi(\cstu(X))$ is a hereditary subalgebra of $\cstu(Y)$, then $\Phi$ is strongly continuous and rank-preserving (\cite[Lemma 6.1]{BragaFarahVignati2019Comm}). We will prove that if there is a rank-preserving embedding then there is a strongly continuous rank-preserving embedding.\footnote{Again see \cite[Theorem~6.1]{BaudierBragaFarahVignatiWillett2022vNA} for the proper context.} Suppose that $\Phi$ is rank-preserving, and let 
\[
q=\SOTh\lim_n\Phi(\chi_{\{x_1,\ldots, x_n\}})
\]
and notice that, by \cite[Theorem 4.3]{BragaFarahVignati2019Comm}, $q\in \cstu(Y)$ and the map
\begin{align*}
 \Phi^q\colon \cstu(X)&\to \cstu(Y)\\
 a&\mapsto q\Phi(a)q
\end{align*} 
is a strongly continuous rank-preserving embedding. As $q\geq\Phi(\chi_x)$, $\Phi^q(\chi_x)=\Phi(\chi_x)$ for all $x\in X$. 

Therefore $\Phi^q$ is a strongly continuous rank-preserving embedding, and there is no loss of generality to assume that $\Phi$ is strongly continuous. Therefore, as~$\Phi$ is rank-preserving, Lemmas \ref{LemmaUnifApprox} and \ref{LemmaGeoCondForEquiApproxAreBoundedBelow} say that if there are no ghost block-rank-one projections in $\cstu(Y)$, then $\inf_{x\in X}\sup_{y\in Y}\|\Phi(\chi_x)\delta_{y}\|>0$. This shows that $\neg$\eqref{Item1.PropMethodFailsForEmbHer}$\Rightarrow\neg$\eqref{Item2.PropMethodFailsForEmbHer} and $\neg$\eqref{Item1.PropMethodFailsForEmbHer}$\Rightarrow\neg$\eqref{Item3.PropMethodFailsForEmbHer}.
\end{proof}

\begin{remark}
As shown in \cite[Corollary C]{LiSpakulaZhang2020}, Item \eqref{Item1.PropMethodFailsForEmbHer} of Proposition \ref{PropMethodFailsForEmbHer} is equivalent to $Y$ containing no coarse disjoint union $Y'=\bigsqcup_nY_n$ consisting of ghostly measured asymptotic expanders (see \cite[Definition 1.3]{LiSpakulaZhang2020} for definitions).
\end{remark}

\section{Are (I) and (II) equivalent?}
As mentioned earlier, the question whether the properties (I) and (II) of $\cstu(X)$ discussed in the introduction are equivalent reduces to an innocent-looking question. 

\begin{question} \label{Q1} Let $X$ be a u.l.f.\ metric space. Suppose that $p_n$ is a sequence of projections in $\cstu(X)$ such that $\SOTh\sum_{n\in M} p_n$ is in $\cstu(X)$ for every $M\subseteq \bbN$ and each $p_n$ is of rank strictly greater than $1$.
\begin{enumerate}
\item Can we conclude that for every $n$ there are projections $q_n\leq p_n$ of rank 1 such that $\SOTh\sum_{n\in M} q_n$ is in $\cstu(X)$ for every $M\subseteq \bbN$? 

\item\label{Q1.2} Can we at least conclude that there are an infinite $M\subseteq \bbN$ and $q_n\leq p_n$ of rank 1 for all $n\in M$ such that $\SOTh\sum_{n\in M_0} q_n$ is in $\cstu(X)$ for every $M_0\subseteq M$?
\end{enumerate}
\end{question}

A positive answer to either part of Question~\ref{Q1} would imply that (I) and (II) are equivalent. A partial result is given in \cite[Lemma 4.3]{BragaFarahVignati2018}, where a positive answer to \eqref{Q1.2} is given in case the projections of interest satisfy our usual regularity condition, that is, if $\inf_{n\in\N}\sup_{x\in X}\norm{p_n\chi_x}>0$.

\begin{acknowledgments}
This paper was written under the auspices of the American Institute of Mathematics (AIM) SQuaREs program and as spinoff of the `Expanders, ghosts, and Roe algebras' SQuaRE project, joint with Florent Baudier, Anna Khukhro, and Rufus Willett.  B.\ M.\ B. was partially supported by the US National Science Foundation under the grant DMS-2054860. I.\ F.\ is partially supported by NSERC. A.\ V.\ is supported by an `Emergence en Recherche' IdeX grant from the Universit\'e Paris Cit\'e and an ANR grant (ANR-17-CE40-0026). The authors would also like to thank Rufus Willett for his help proofreading this paper.
\end{acknowledgments}
 
\bibliographystyle{amsalpha}
\bibliography{bibliography}

\newcommand{\etalchar}[1]{$^{#1}$}
\providecommand{\bysame}{\leavevmode\hbox to3em{\hrulefill}\thinspace}
\providecommand{\MR}{\relax\ifhmode\unskip\space\fi MR }
\providecommand{\MRhref}[2]{%
  \href{http://www.ams.org/mathscinet-getitem?mr=#1}{#2}
}
\providecommand{\href}[2]{#2}
\begin{thebibliography}{{Bra}22}

\bibitem[BBF{\etalchar{+}}a]{BaudierBragaFarahKhukhroVignatiWillett2021uRaRig}
F.~Baudier, B.~M. {Braga}, I.~{Farah}, A.~Khukhro, A.~{Vignati}, and
  R.~Willett, \emph{{Uniform Roe algebras of uniformly locally finite metric
  spaces are coarsely rigid}}, To appear in Inventiones Mathematicae.

\bibitem[BBF{\etalchar{+}}b]{BaudierBragaFarahVignatiWillett2022vNA}
F.~Baudier, B.~M. {Braga}, I.~{Farah}, A.~{Vignati}, and R.~Willett,
  \emph{{Embeddings of von Neumann algebras into uniform Roe algebras}}, in
  preparation.

\bibitem[BF21]{BragaFarah2018Trans}
B.~M. Braga and I.~Farah, \emph{On the rigidity of uniform {R}oe algebras over
  uniformly locally finite coarse spaces}, Trans. Amer. Math. Soc. \textbf{374}
  (2021), no.~2, 1007--1040. \MR{4196385}

\bibitem[BFV20]{BragaFarahVignati2019Comm}
B.~M. Braga, I.~Farah, and A.~Vignati, \emph{Embeddings of uniform {R}oe
  algebras}, Comm. Math. Phys. \textbf{377} (2020), no.~3, 1853--1882.
  \MR{4121613}

\bibitem[BFV21]{BragaFarahVignati2018}
\bysame, \emph{Uniform {R}oe coronas}, Adv. Math. \textbf{389} (2021), Paper
  No. 107886, 35. \MR{4288216}

\bibitem[BFV22]{BragaFarahVignati2020AnnInstFour}
\bysame, \emph{General uniform {R}oe algebra rigidity}, Ann. Inst. Fourier
  (Grenoble) \textbf{72} (2022), no.~1, 301--337. \MR{4448597}

\bibitem[{Bra}22]{Braga2020CoarseQuotient}
B.~M. {Braga}, \emph{{Coarse quotients of metric spaces and embeddings of
  uniform Roe algebras}}, {To appear in the Journal of Noncommutative Geoemtry}
  (2022), arXiv:2009.06794.

\bibitem[Far19]{Fa:STCstar}
I.~Farah, \emph{Combinatorial set theory and \cstar-algebras}, Springer
  Monographs in Mathematics, Springer, 2019.

\bibitem[HLS02]{HigsonLafforgueSkandalis2002GAFA}
N.~Higson, V.~Lafforgue, and G.~Skandalis, \emph{Counterexamples to the
  {B}aum--{C}onnes conjecture}, Geom. Funct. Anal. \textbf{12} (2002), no.~2,
  330--354.

\bibitem[L{\v{S}}Z20]{LiSpakulaZhang2020}
K.~{Li}, J.~{{\v{S}}pakula}, and J.~{Zhang}, \emph{{Measured asymptotic
  expanders and rigidity for Roe algebras}}, arXiv e-prints (2020),
  arXiv:2010.10749.

\bibitem[Lub12]{lubotzky2012expander}
A.~Lubotzky, \emph{Expander graphs in pure and applied mathematics}, Bull.
  Amer. Math. Soc. \textbf{49} (2012), no.~1, 113--162.

\bibitem[Roe03]{RoeBook}
J.~Roe, \emph{Lectures on coarse geometry}, University Lecture Series, vol.~31,
  American Mathematical Society, Providence, RI, 2003. \MR{2007488}

\bibitem[RW14]{RoeWillett2014}
J.~Roe and R.~Willett, \emph{Ghostbusting and property {A}}, J. Funct. Anal.
  \textbf{266} (2014), no.~3, 1674--1684. \MR{3146831}

\bibitem[Sak14]{Sako2014}
H.~Sako, \emph{Property {A} and the operator norm localization property for
  discrete metric spaces}, J. Reine Angew. Math. \textbf{690} (2014), 207--216.
  \MR{3200343}

\bibitem[{\v{S}}W13]{SpakulaWillett2013AdvMath}
J.~{\v{S}}pakula and R.~Willett, \emph{On rigidity of {R}oe algebras}, Adv.
  Math. \textbf{249} (2013), 289--310. \MR{3116573}

\bibitem[WW20]{WhiteWillett2017}
S.~White and R.~Willett, \emph{Cartan subalgebras in uniform {R}oe algebras},
  Groups Geom. Dyn. \textbf{14} (2020), no.~3, 949--989. \MR{4167028}

\end{thebibliography}

\end{document}